\documentclass[12pt]{amsart}
\usepackage{amsmath,amsthm,amsfonts,amssymb,amscd,flafter,pinlabel}
\usepackage[mathscr]{eucal}
\usepackage{graphics}
 \usepackage[all,knot,arc]{xy}
\usepackage{bbm}

\usepackage[usenames,dvipsnames]{color}

\usepackage{graphicx}

\usepackage{pgf}
\usepackage{tikz}
\usetikzlibrary{arrows,automata}
\usepackage[latin1]{inputenc}
\usepackage{enumitem}

\usepackage[colorlinks=true,linkcolor=blue,citecolor=blue,urlcolor=blue]{hyperref}

\usepackage[pdftex,lmargin=1in,rmargin=1in,tmargin=1in,bmargin=1in]{geometry}

\newtheorem{theorem}{Theorem}[section]

\theoremstyle{definition}
\newtheorem{definition}{Definition}[section]
\theoremstyle{definition}
\newtheorem{remark}[theorem]{Remark}

\newcommand{\R}{\ensuremath{\mathbb{R}}}
\newcommand{\Z}{\ensuremath{\mathbb{Z}}}

\newcommand\F{\mathbb F}

\def\X{\mathbb{X}}
\def\O{\mathbb{O}}

\DeclareMathOperator{\id}{id}
\newcommand{\bdy}{\partial}

\newcommand{\ainf}{\mathcal{A}_\infty}

\newcommand{\zzz}{\mathbf{z}}
\newcommand{\xx}{{\bf x}}
\newcommand{\yy}{{\bf y}}

\newcommand{\T}{\mathcal{T}}

\newcommand{\OO}{\mathbb O}
\newcommand{\XX}{\mathbb X}

\newcommand{\balpha}{\boldsymbol\alpha}
\newcommand{\bbeta}{\boldsymbol\beta}
\newcommand{\bgamma}{\boldsymbol\gamma}

\def\Aa{\mathcal{A}}
\def\P{\mathcal{P}}
\newcommand{\HH}{\mathcal{H}}

\def\SO{S_\O}
\def\SX{S_\X}

\def\CDA^-{\mathit{CDA}^-}

\def\AA{\ensuremath{\mathit{AA}}}

\def\DA{\ensuremath{\mathit{DA}}}

\def\bar{\overline}

\newcommand{\hf}{\mathit{HF}}
\newcommand{\cfk}{\mathit{CFK}}
\newcommand{\hfhat}{\widehat{\hf}}
\newcommand{\cfkhat}{\widehat{\cfk}}
\newcommand{\hfk}{\mathit{HFK}}
\newcommand{\ct}{\mathit{CT}}
\newcommand{\ctt}{\widetilde{\ct}}
\newcommand{\hc}{\HH^{\circ}}
\newcommand{\xc}{\xx^{\circ}}
\newcommand{\yc}{\yy^{\circ}}

\newcommand\blfootnote[1]{%
  \begingroup
  \renewcommand\thefootnote{}\footnote{#1}%
  \addtocounter{footnote}{-1}%
  \endgroup
}

\begin{document}

\title{A self-pairing theorem for tangle Floer homology}

\author{Ina Petkova}
\address {Department of Mathematics, Rice University\\ Houston, TX 77005}
\email {ina@rice.edu}
\urladdr{\href{http://www.math.rice.edu/~tvp}{http://www.math.rice.edu/\~{}tvp1}}
\author[Vera V\'ertesi]{Vera V\'ertesi}
\address{Institut de Recherche Math\'ematique Avanc\'ee \\Universit\'e de Strasbourg}
\email{vertesi@math.unistra.fr}

\maketitle

\begin{abstract}
We show that for a tangle $T$ with $-\bdy^0T \cong \bdy^1 T$ the Hochschild homology of the tangle Floer homology $\ctt(T)$ is equivalent to the link Floer homology of the closure $T' = T/(-\bdy^0T \sim \bdy^1 T)$ of the tangle, linked with  the tangle axis. In addition, we show that the action of the braid group on tangle Floer homology is faithful.
\end{abstract}

\section{introduction}

\blfootnote{IP received support from an AMS-Simons travel grant. VV was supported by ERC Geodycon, OTKA grant number NK81203 and NSF grant number 1104690.}

Tangle Floer homology is an invariant of tangles in $3$-manifolds with  boundary  $S^2$ or $S^2\coprod S^2$, or in closed $3$-manifolds, which  takes the form of a differential graded module, bimodule, or a chain complex, respectively \cite{pv}.  It behaves well under gluing and recovers knot Floer homology. 
Before we state the main results, we recall some  definitions from \cite{pv} and make some new ones. 

\begin{definition}\label{def:sphere}
An \emph{$n$-marked sphere} $\mathcal S$ is a sphere $S^2$ with $n$ oriented points $t_1,\dots,t_n$  on its equator $S^1\subset S^2$ numbered respecting the orientation of $S^1$.
\end{definition}

\begin{definition}\label{tangle_def}
A \emph{marked $(m,n)$-tangle} $\T$ in an oriented $3$-manifold $Y$ with two boundary components  $\bdy^0 Y\cong S^2$ and $\bdy^1
Y\cong S^2$ is a properly embedded 1--manifold $T$ with $(-\bdy^0
Y,-(\bdy^0 Y\cap\bdy T))$ identified with an $m$-marked sphere and $(\bdy^1 Y,\bdy^1
Y\cap\bdy T)$ identified with an
$n$-marked sphere (via orientation-preserving diffeomorphisms).
We denote $-(\bdy^0 Y\cap\bdy T)$ and $\bdy^1 Y\cap\bdy T$ along with the ordering information by $-\bdy^0\T$ and  $\bdy^1\T$.
\end{definition}

\begin{definition}\label{stangle_def}
A \emph{strongly marked $(m,n)$-tangle $(Y, \T, \gamma)$} is a marked $(m,n)$-tangle $(Y, \T)$, along with a framed arc $\gamma$  connecting $\bdy^0 Y$ to $\bdy^1 Y$ in the complement of $\T$ such that  $\gamma$ and its framing $\lambda_{\gamma}$ (viewed as a push off of $\gamma$) have ends on the equators of the two marked spheres,  and we see $-\bdy^0\T,  -\bdy^0\gamma, -\bdy^0\lambda_{\gamma}$ and $\bdy^1\T,  \bdy^1\gamma, \bdy^1\lambda_{\gamma}$ in this order along each equator. See Figure \ref{fig:gen_tc}.
\end{definition}

As a special case, an $(m,n)$-tangle in $\R^2\times I$ is a cobordism (contained in $[1,\infty)\times \R\times I$) from $\{1, \ldots, m\}\times \{0\}\times \{0\}$ to $\{1, \ldots, n\}\times \{0\}\times \{1\}$. A tangle in $\R^2\times I$ can be thought of as a strongly marked tangle, by compactifying $\R^2\times I$ to $S^2\times I$, taking the images of $\R\times \{0\}\times \{0\}$ and $\R\times \{0\}\times \{1\}$ to be the equators of the marked spheres, and setting $(\gamma, \lambda_{\gamma}):=(\{(-1,0)\}\times I, \{(0,0)\}\times I)$.

We turn our attention to  strongly marked tangles $(Y, \T, \gamma)$ with $-\bdy^0\T \cong \bdy^1 \T$. 

\begin{definition}\label{def:c} 
A strongly marked tangle $(Y, \T, \gamma)$ is called \emph{closable} if $-\bdy^0\T \cong \bdy^1\T$. Given a closable tangle $(Y, \T, \gamma)$, we can ``glue it to itself" to form its \emph{closure} $(T', Y', \gamma')$ by identifying the two boundary components of $Y$,  $-\bdy^0Y$ and $\bdy^1Y$, with the same marked sphere. The \emph{surgered closure} of  $(Y, \T, \gamma)$ is the pair $(T_0, Y_0)$, where the link $T_0$ is the union of $T'$ and the negatively oriented meridian $\mu_{\gamma}$ of $\gamma$ in the $0$-surgery $Y_0 = Y'_0(\gamma')$ of $Y'$ along the framed knot $\gamma'$. We call $\mu_{\gamma}$ the \emph{tangle axis} of the tangle $\T$. See Figure \ref{fig:gen_tc}.
\end{definition}

\begin{figure}[h]
 \centering
      \labellist   
        \pinlabel $\bdy^0Y$ at -20 135 
        \pinlabel $T$ at 45 70 
        \pinlabel $\dots$ at 50 120 
        \pinlabel $\dots$ at 50 18 
        \pinlabel $\dots$ at 275 120 
        \pinlabel $\dots$ at 275 18 
        \pinlabel $\bdy^1Y$ at -20 8 
        \pinlabel \textcolor{green}{$\gamma$} at 100 70 
        \pinlabel \textcolor{green}{$\lambda_{\gamma}$} at 122 70 
        \pinlabel $T$ at 270 70 
        \pinlabel $0$ at 346 110 
       \pinlabel $\mu_{\gamma}$ at 360 70 
     \endlabellist
       \includegraphics[scale=.7]{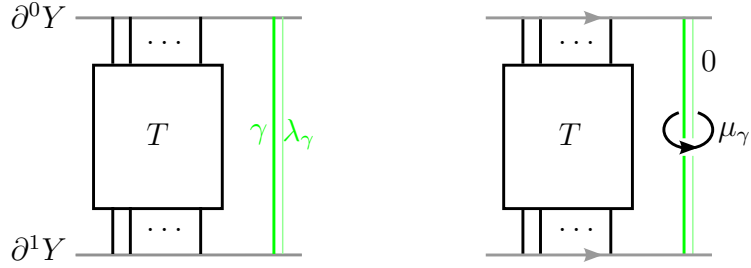} 
       \vskip .2 cm
       \caption{Left: A strongly marked tangle $(T, Y, \gamma)$. Right: The surgered closure $(T_0, Y_0)$ of the tangle  $(T, Y, \gamma)$.}\label{fig:gen_tc}
\end{figure}

When $Y$ is $S^2\times I$ and $(\gamma, \lambda_{\gamma})$ is a product as above,  then $Y_0\cong S^3$ and  $T_0$ is the link formed by  the closure of $T\subset \R^2\times I\subset S^3$ and an unknot that is the boundary of a disk containing $-\bdy^0T \sim \bdy^1T$, see Figure \ref{fig:tc}.
For example, for a braid $B\in \R^2\times I$, the tangle axis is precisely the braid axis. 

\begin{figure}[h]
 \centering
      \labellist   
        \pinlabel $T$ at 55 120 
        \pinlabel $\dots$ at 60 175 
        \pinlabel $\dots$ at 60 70 
        \pinlabel $\dots$ at 332 175 
        \pinlabel $\dots$ at 332 70 
        \pinlabel $T$ at 328 120
     \endlabellist
       \includegraphics[scale=.6]{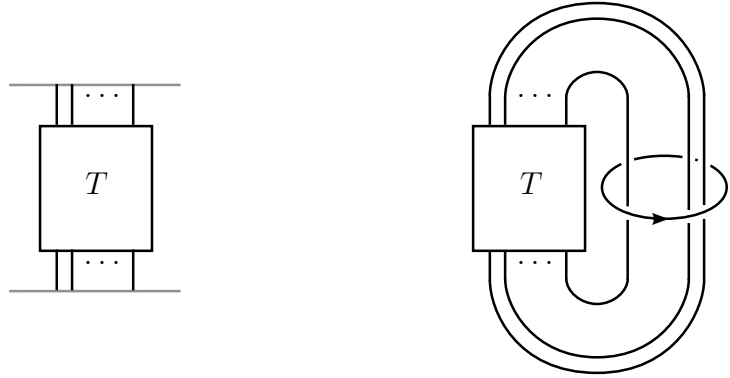} 
       \caption{Left: A tangle $T\subset \R^2\times I$. Right: The corresponding link $T_0\subset S^3$.}\label{fig:tc}
\end{figure}

For a tangle $(Y, \T, \gamma)$,  the tangle Floer homology $\ctt(Y, \T, \gamma)$ is a left-right $\DA$ bimodule over $(\Aa(-\bdy^0\T), \Aa(\bdy^1\T))$, where $\Aa(-\bdy^0\T)$ and $\Aa(\bdy^1\T)$ are differential graded algebras associated to $-\bdy^0\T$ and  $\bdy^1\T$, respectively,  see \cite{pv}.  For a closable tangle, these two algebras are the same, and one can take the Hochschild homology of the bimodule, see \cite[Section 2.3.5]{bimod}. 
We show that this Hochschild homology is the knot Floer homology of the surgered closure of the tangle (see Theorems \ref{thm:gen} and \ref{thm:graded} for precise statements). 

\newtheorem{thmm}{Theorem}

\begin{thmm}\label{thm:main}
Let $(Y, \T, \gamma)$ be a closable strongly marked tangle. Then there is an equivalence
\[\mathit{HH}(\ctt(T, Y))\cong \widetilde\hfk(T_0, Y_{0(\gamma)}).\]
\end{thmm}

Combined with a result of Baldwin-Grigsby \cite{bg}, we get the following corollary.

\newtheorem{corm}[thmm]{Corollary}
\begin{corm}\label{cor:braid}
Tangle Floer homology restricts to a faithful linear-categorical action of the braid group.
\end{corm}

\subsection*{Acknowledgments}
We thank Robert Lipshitz for a helpful conversation,  the referee for suggesting Corollary \ref{cor:braid}, and Mohammed Abouzaid and Ailsa Keating for a useful discussion of $A_{\infty}$-categories.
\section{Algebra review}

Let $A$ be a unital differential graded algebra over a ground ring ${\bf k}$, where  ${\bf k}$ is a direct sum of copies of $\F_2 = \Z/2\Z$.  The unit gives a preferred map $\iota:{\bf k}\to A$.
We assume that  $A$ is \emph{augmented}, i.e. there is a map $\epsilon:A\to {\bf k}$ such that $\epsilon(1) =1$, $\epsilon(ab) = \epsilon(a)\epsilon(b)$, and $\epsilon(\bdy a) = 0$. The \emph{augmentation ideal} $\ker\epsilon$ is denoted by $A_+$.

A \emph{type $\DA$ bimodule over $(A,A)$} is a graded  ${\bf k}$-bimodule $N$, together with degree $0$, ${\bf k}$-linear maps
\[\delta^1_{1+j}: N\otimes A[1]^{\otimes j}\to A\otimes N[1],\]
satisfying a certain compatibility condition, see \cite[Definition 2.2.42]{bimod}. 

A $\DA$ bimodule is \emph{bounded} if the structure maps behave in a certain nice way, see \cite[Definition 2.2.45]{bimod}. We will not recall the complete definition of boundedness here, but we point out that the structures arising from nice Heegaard diagrams are bounded, and moreover the only nonzero structure maps in that case are $\delta_1^1$ and $\delta_2^1$. We will call a $\DA$ bimodule \emph{nice} if it is bounded and $\delta^1_i = 0$ for all $i>2$.

Given a bounded type $\DA$ bimodule $N$ over $(A, A)$, one can define a chain complex $(N^{\circ}, \widetilde\bdy)$ whose homology agrees with the Hochschild homology of the $\ainf$-bimodule $A\boxtimes N$ corresponding to $N$, see \cite[Section 2.3.5]{bimod}. The vector space $N^{\circ}$, called the the \emph{cyclicization}  of $N$,  is  the quotient $N/[N, {\bf k}]$, where $[N, {\bf k}]$ is the submodule of $N$ generated by elements $xk-kx$, for $x\in N$ and $k\in {\bf k}$. The differential $\widetilde\bdy$ is easy to describe when $N$ is nice. We recall the construction in this special case below.

Define a cyclic rotation map $R: (A\otimes N)^{\circ}\to (N\otimes A_+)^{\circ}$
by 
\[R(a\otimes x) = x\otimes [(\id-\iota\circ\epsilon)(a)]\]
The map 
$\epsilon\otimes \id:A\otimes N\to {\bf k}\otimes N = N$
descends to a map $(A\otimes N)^{\circ}\to N^{\circ}$, which we will also denote $\epsilon$. We denote the cyclicizations of $\delta_1^1:N\to A\otimes N$ and $\delta_2^1:N\otimes A_+\to A\otimes N$ by  $\delta_1^1$ and $\delta_2^1$ as well. 
Finally, $\widetilde\bdy$ is defined as 
$$\widetilde\bdy =\epsilon\circ\delta_1^1 + \epsilon\circ \delta_2^1\circ R\circ \delta _1^1.$$

Given a tangle $(Y, \T, \gamma)$, one can represent it by a multipointed bordered Heegaard diagram $\HH$ with two boundary components $\bdy^0\HH$ and $\bdy^1\HH$, see \cite[Section 8.2]{pv}. To $-\bdy^0\HH$ and $\bdy^1\HH$ one associates differential algebras $\Aa(-\bdy^0\T)$ and $\Aa(\bdy^1\T)$, and to $\HH$ a $\DA$ bimodule $\ctt(\HH)$ over $\Aa(-\bdy^0\T)$ and $\Aa(\bdy^1\T)$. The structure maps on the bimodule are obtained by counting certain holomorphic curves in  $\HH\times I\times \R$. See \cite[Sections 7.2 and 10.3]{pv} For a tangle in $\R^2\times I$, the bimodule can also  be defined in terms of sequences of strand diagrams corresponding to a decomposition of the tangle into elementary pieces, see \cite[Sections 3 and 5.2]{pv}. We do not recall the two constructions here, but refer the reader to  \cite{pv}.

\section{Proofs of the main results}\label{sec:proof}

We provide the full statements of the main theorem in the general case (ungraded), and in the case of $\R^2\times I$, and  prove both  via nice diagrams.

Tangle Floer homology is an invariant of the tangle---if $\HH_1$ and $\HH_2$ are Heegaard diagrams for $(Y, \T, \gamma)$ with $2k_1$ and $2k_2$ basepoints, respectively, and $k_1\geq k_2$, then $\ctt(\HH_1)\simeq\ctt(\HH_2)\otimes (\F_2\oplus\F_2)^{\otimes(k_1-k_2)}$.  For a closable tangle, $-\bdy^0\T\cong \bdy^1\T$, so  the algebras $\Aa(-\bdy^0\T)$ and $\Aa(\bdy^1\T)$ are the same, and one can take the Hochschild homology of the bimodule. 

\begin{theorem}\label{thm:gen}
If $\HH$ is Heegaard diagram with $2k$ basepoints for a  closable strongly marked tangle $(Y, \T, \gamma)$, and  $\HH^{\circ}$ is a diagram with $2k+4$ basepoints for $(T_0, Y_0)$, for some $k$, then
\[\mathit{HH}(\ctt(\HH))\cong \widetilde\hfk(\HH^{\circ}).\]
\end{theorem}

Note that in \cite{pv} gradings for tangle Floer homology are only defined when the underlying manifold is $S^2\times I$ or  $B^3$, so Theorem \ref{thm:gen} only claims an ungraded isomorphism.

\begin{proof} The proof is very similar to  that of \cite[Theorem 14]{bimod}.
By invariance under Heegaard moves, it suffices to prove the theorem for one choice of  $\HH$,  and one choice of $\HH^{\circ}$ with $4$ more basepoints than $\HH$. 

Let $(Y, \T, \gamma)$ be a closable strongly marked $(n,n)$-tangle, and let $\HH= (\Sigma, \balpha, \bbeta, \XX, \OO, \zzz)$ be a nice bordered Heegaard diagram for  $(Y, \T, \gamma)$, as in \cite[Proposition 12.1]{pv}. Glue $\HH$ to itself  by identifying $-\bdy^0\HH$ and $\bdy^1\HH$, and call the result $\HH'$ (note this is not a valid Heegaard diagram).  Recall that $\zzz = \{\zzz_1, \zzz_2\}$ is a set of two arcs in $\Sigma\setminus (\balpha\cup \bbeta)$ with boundary on  $\bdy\Sigma \setminus \balpha$, oriented from the left to the right boundary, and  let $\zzz_1'$ and  $\zzz_2'$ be the resulting closed curves in $\HH'$. Surger $\HH'$ along $\zzz_1'$ and  $\zzz_2'$, and place $4$ basepoints in the $4$ resulting regions: $X_1$, $O_1$, $X_2$, and $O_2$ in the region whose boundary contains $a_{n+1}^0$,$a_1^0$, $a_{2n+2}^0$, and $a_{n+2}^0$, respectively. The result is a diagram  $\HH^{\circ} = (\Sigma^{\circ}, \balpha^{\circ}, \bbeta^{\circ}, \XX^{\circ}, \OO^{\circ})$, see Figure \ref{fig:generic}. 

\begin{figure}[h]
 \centering
      \labellist   
       \pinlabel $\textcolor{red}{a_1^0}$ at -15 160
       \pinlabel $\textcolor{red}{a_{n+1}^0}$ at -20 215  
       \pinlabel $\textcolor{red}{a_{n+2}^0}$ at -22 40
       \pinlabel $\textcolor{red}{a_{2n+2}^0}$ at -27 86  
       \pinlabel $\textcolor{red}{a_1^1}$ at 252 160
       \pinlabel $\textcolor{red}{a_{n+1}^1}$ at 262 215  
       \pinlabel $\textcolor{red}{a_{n+2}^1}$ at 261 40
       \pinlabel $\textcolor{red}{a_{2n+2}^1}$ at 265 86  
       \pinlabel $\textcolor{green}{\zzz_1}$ at 120 110
       \pinlabel $\textcolor{green}{\zzz_2}$ at 120 236
        \pinlabel \textcolor{gray}{$A$} at 120 185
        \pinlabel \textcolor{gray}{$B$} at 120 60 
        \pinlabel \textcolor{gray}{$A$} at 420 185
        \pinlabel \textcolor{gray}{$B$} at 420 60 
        \pinlabel \textcolor{gray}{$A$} at 685 185
        \pinlabel \textcolor{gray}{$B$} at 685 60 
        \pinlabel {\footnotesize $X_1$} at 784 220 
        \pinlabel {\footnotesize  $O_1$} at 784 155 
        \pinlabel {\footnotesize  $X_2$} at 784 92 
        \pinlabel {\footnotesize  $O_2$} at 784 29
     \endlabellist
      \hspace{.7cm} \includegraphics[scale=.55]{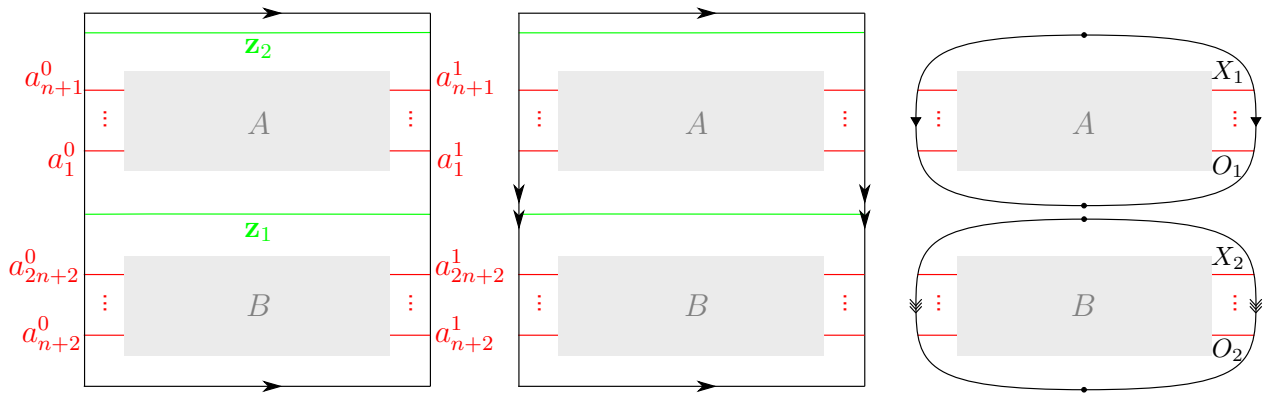} 
       \vskip .1 cm 
       \caption{Left: A Heegaard diagram $\HH$ for a closable tangle $(Y, \T, \gamma)$. The left edge is $\bdy^0\HH$ and the right edge is $\bdy^1\HH$. Middle: The  diagram $\HH'$ obtained by self-gluing $\HH$. Right: The Heegaard diagram $\HH^{\circ}$ obtained by surgery on the two green curves on $\HH'$.}\label{fig:generic}
\end{figure}

Recall that originally $(Y,\T)$ was obtained by attaching a two handle to the pair $(Y_{\mathit{dr}},\T_{\mathit{dr}})$ that we got from a drilled diagram on $\Sigma-N(\zzz_2)$ by the usual handle attachments. Thus, the Heegaard diagram $\HH''$ that we get by only doing surgery on $\HH'$ along $\zzz_2'$ describes $(Y'-N(\gamma'),T')$, with $\zzz_1'$ lying on the boundary of $N(\gamma')$. Now the surgery along $\zzz_1'$ on $\HH''$ simply results in the Dehn filling of $(Y'-N(\gamma'),T')$ with framing $\zzz_1'$. The knot described by $X_1$, $O_1$, $X_2$, and $O_2$ goes through the former $\zzz_2$ and $\zzz_1$ on the Heegaard diagram, thus indeed describes a meridian for $T'$. This means that $\hc$ is a Heegaard diagram for $(T_0, Y_0)$.  

Observe that the generators of $\HH'$, or equivalently the generators of $\HH^{\circ}$,  correspond to generators $\xx$ of $\HH$ with $\bar o^0(\xx) = o^1(\xx)$, where as in \cite{pv}  $o^1(\xx)$ (respectively $\bar o^0(\xx)$) denotes the set of $\balpha$ arcs that are occupied (not occupied) by $\xx$ on $\balpha^1$ (and $\balpha^0$).

Denote the algebra $\Aa(\bdy^1\T)\cong \Aa(-\bdy^0\T)$ by $\Aa$, and its ring of idempotents by ${\bf k}$. Recall that $\Aa$ has a basis over $\F_2$ consisting of strand diagrams \cite[Section 7]{pv}. 
We define the augmentation map $\epsilon:\Aa\to {\bf k}$ on this basis explicitly: it is the identity on generators in  ${\bf k}\subset \Aa$ and zero on generators $a\notin{\bf k}\subset \Aa$.
The structure maps on the $\DA$ bimodule $\ctt(\HH)$ count the following types of domains (see \cite[Sections 10 and 12]{pv}):
\begin{enumerate}
\item empty provincial rectangles and bigons. These contribute to $\delta_1^1$,  with image in ${\bf k}\otimes \ctt(\HH)\subset \Aa\otimes \ctt(\HH)$. 
\item empty rectangles that intersect $\bdy^0\HH$ (the left boundary of $\HH$). These  contribute to $\delta_1^1$, with image in $\Aa_+\otimes \ctt(\HH)\subset \Aa\otimes \ctt(\HH)$. 
\item sets of empty rectangles, each of which  intersects $\bdy^1\HH$ (the right boundary of $\HH$). These 
comprise $\delta_2^1$, whose image is entirely contained in ${\bf k}\otimes \ctt(\HH)$.
\end{enumerate}

The differential on the Hochschild complex $(\ctt(\HH)^{\circ}, \widetilde\bdy)$ then counts the following domains on $\HH$. The map $\epsilon\circ\delta_1^1$ counts provincial rectangles and bigons, and then forgets the idempotent component of the output. These are exactly the empty rectangles and bigons in $\HH^{\circ}$ that do not cross $-\bdy^0\HH = \bdy^1\HH$.  The map $\epsilon\circ \delta_2^1\circ R\circ \delta _1^1$ counts rectangles as follows. Since the rotation map $R$ is zero on elements $e\otimes x$ with $e\in {\bf k}$, only the part of $\delta_1^1$ that counts domains of Type $(2)$ contributes. Thus, the image of $R\circ \delta_1^1$ is generated by elements of form $y\otimes a$, where $a\in \Aa_+$ is a generator with only one moving strand.  Thus, the part of $\delta_2^1$ that contributes to $\epsilon\circ \delta_2^1\circ R\circ \delta _1^1$ counts individual empty rectangles that intersect $\bdy^1\HH$. To sum up,  $\epsilon\circ \delta_2^1\circ R\circ \delta _1^1$ counts pairs of empty rectangles, one with an edge on $\bdy^0\HH$ and one with an edge on $\bdy^1\HH$, which glue up to a rectangle after the identification $-\bdy^0\HH \sim \bdy^1\HH$. These are exactly the empty rectangles in $\HH^{\circ}$ that cross $-\bdy^0\HH \sim \bdy^1\HH$.

Thus, $(\ctt(\HH)^{\circ}, \widetilde\bdy)\cong \widetilde\cfk(\HH^{\circ})$.
\end{proof}

For the special case of a tangle $T$ in $\R^2\times I$, we state a graded version. In this case,  $\mathit{HH}(\ctt(\HH))$ inherits the Maslov and Alexander gradings $M$ and $A$ from $\ctt(\HH)$, and also carries a \emph{strands} grading $S$ (counting the number of occupied $\alpha$-arcs that touch $\bdy^1\HH$). 
If $\HH_1$ and $\HH_2$ are Heegaard diagrams for $T$ with $2k_1$ and $2k_2$ basepoints, respectively, and $k_1\geq k_2$, then $\ctt(\HH_1)\simeq\ctt(\HH_2)\otimes (\F_2\oplus\F_2)^{\otimes(k_1-k_2)}$, where each tensor factor $\F_2\oplus\F_2$ has one summand in $(M, A, S)$ trigrading $(0,0, 0)$ and one in trigrading $(-1,-1, 0)$.

Let $l$ be the number of components of  $T_0$, and label the components $L_0 =\mu_{\gamma}, L_1, \ldots, L_{l-1}$. The link Floer homology of the Heegaard diagram $\HH^{\circ}$ for  $T_0$ described earlier, $\widetilde{\mathit{HFL}}(\HH^{\circ})$,  is multigraded, with Maslov grading $M$ in $\Z + \frac{l-1}{2}$ and Alexander multigrading $(A_0, \ldots, A_{l-1})$ in $(\frac 1 2 \Z)^l$, with each  $\frac 1 2 \Z$ factor corresponding to a component of the link  \cite{oszlink}. Let $\widetilde{\mathit{HFL}}(\HH^{\circ}, 0)$ be $\widetilde{\mathit{HFL}}(\HH^{\circ})$ with multigrading collapsed to a trigrading by $M$, $A_0$, and $A' := A_1+\cdots A_{l-1} = A-A_0$ (here $A = A_0+\cdots+A_{l-1}$ is the Alexander grading on $\widetilde{\mathit{HFK}}(\HH^{\circ})$).

\begin{theorem}\label{thm:graded}
Let $\HH$ be a Heegaard diagram with $2k$ basepoints for a tangle $T$ in $\R^2\times I$ with  $-\bdy^0T \cong \bdy^1T$, and let $\HH^{\circ}$ be a diagram with $2k+4$ basepoints for the surgered closure $T_0$, with $4$ of the basepoints corresponding to the component $\mu_{\gamma}$. Then there is an isomorphism
\[\mathit{HH}(\ctt(\HH))\cong \widetilde{\mathit{HFL}}(\HH^{\circ},0)\]
which respects the trigrading in the following sense. If the isomorphism maps a homogeneous element $\xx\in \mathit{HH}(\ctt(\HH))$ to an element $\yy\in \widetilde{\mathit{HFL}}(\HH^{\circ}, 0)$, then $\yy$ is homogeneous and 
\begin{align*}
M(\yy) &= M(\xx) + S(\xx)-a-1\\
A'(\yy) &= A(\xx) -S(\xx)+ \frac l 2  +n-a-1\\
A_0(\yy) &=  S(\xx) -\frac{n+1}{2},
\end{align*}
where  $n=|\bdy^1T|$,   $a$ is the number of positively oriented points in $\bdy^1T$, and $l$ is the number of components of  $T_0$.
\end{theorem}

\begin{proof} 
Again, by invariance under Heegaard moves, it suffices to prove the theorem for any one specific choice of $\HH$ and $\HH^{\circ}$ with the prescribed relative number of basepoints. We already discussed the isomorphism in the proof of Theorem \ref{thm:gen}. It remains to identify the gradings. 

Let $\HH$ be a Heegaard diagram for $T$ obtained by plumbing annular bordered grid diagrams, as in \cite[Section 4]{pv}. By gluing on a diagram for the straight strands $\bdy^1T\times I$ if necessary, we may assume that $\HH$ has even genus, which we denote by $2g$ (this makes for an easier gradings argument). See the top diagram in Figure \ref{fig:grid}. 
We modify $\HH$ to a diagram $\HH^{\circ}$ for $T_0$, as in the proof of Theorem \ref{thm:gen}. See the bottom diagram in Figure \ref{fig:grid}. 

\begin{figure}[h!]
 \centering
      \labellist   
       \pinlabel ${\bdy^0 \HH}$ at -25 514
       \pinlabel $\textcolor{red}{a_1^0}$ at -2 555
       \pinlabel $\textcolor{red}{a_{n+1}^0}$ at 15 584  
       \pinlabel $\textcolor[rgb]{1,.5,.5}{a_{n+2}^0}$ at 12 472
       \pinlabel $\textcolor[rgb]{1,.5,.5}{a_{2n+2}^0}$ at -4 446 
       \pinlabel $\textcolor{red}{a_1^1}$ at 518 555
       \pinlabel $\textcolor{red}{a_{n+1}^1}$ at 542 584  
       \pinlabel $\textcolor{red}{a_{n+2}^1}$ at 542 472
       \pinlabel $\textcolor{red}{a_{2n+2}^1}$ at 535 446  
       \pinlabel $\textcolor{green}{\zzz_1}$ at 285 465
       \pinlabel $\textcolor[rgb]{.7,1,.7}{\zzz_2}$ at 285 540
        \pinlabel \textcolor[rgb]{.7,.7,.7}{\scriptsize ${G_1}$} at 76 506
        \pinlabel \textcolor[rgb]{.7,.7,.7}{\scriptsize ${G_2}$} at 115 506 
        \pinlabel \textcolor[rgb]{.7,.7,.7}{\scriptsize ${G_{4g}}$} at 464 506    
        \pinlabel ${\bdy^1 \HH}$ at 562 514
        \pinlabel \textcolor{red}{$\dots$} at 260 120 
        \pinlabel \textcolor{red}{$\dots$} at 260 230 
        \pinlabel \textcolor{red}{$\dots$} at 260 460 
        \pinlabel \textcolor{red}{$\dots$} at 260 570 
        \pinlabel $\dots$ at 260 175 
        \pinlabel \rotatebox{90}{$\boldsymbol{\dots}$} at 132 200 
        \pinlabel \rotatebox{90}{$\boldsymbol{\dots}$} at 308 200 
        \pinlabel \textcolor{gray}{$\dots$} at 260 510 
        \pinlabel \textcolor{gray}{{\small $X_1$}} at 509 256 
        \pinlabel \textcolor{gray}{\small $O_1$} at 486 198 
        \pinlabel {\small $X_2$} at 471 95 
        \pinlabel {\small $O_2$} at 499 147
     \endlabellist
       \includegraphics[scale=.59]{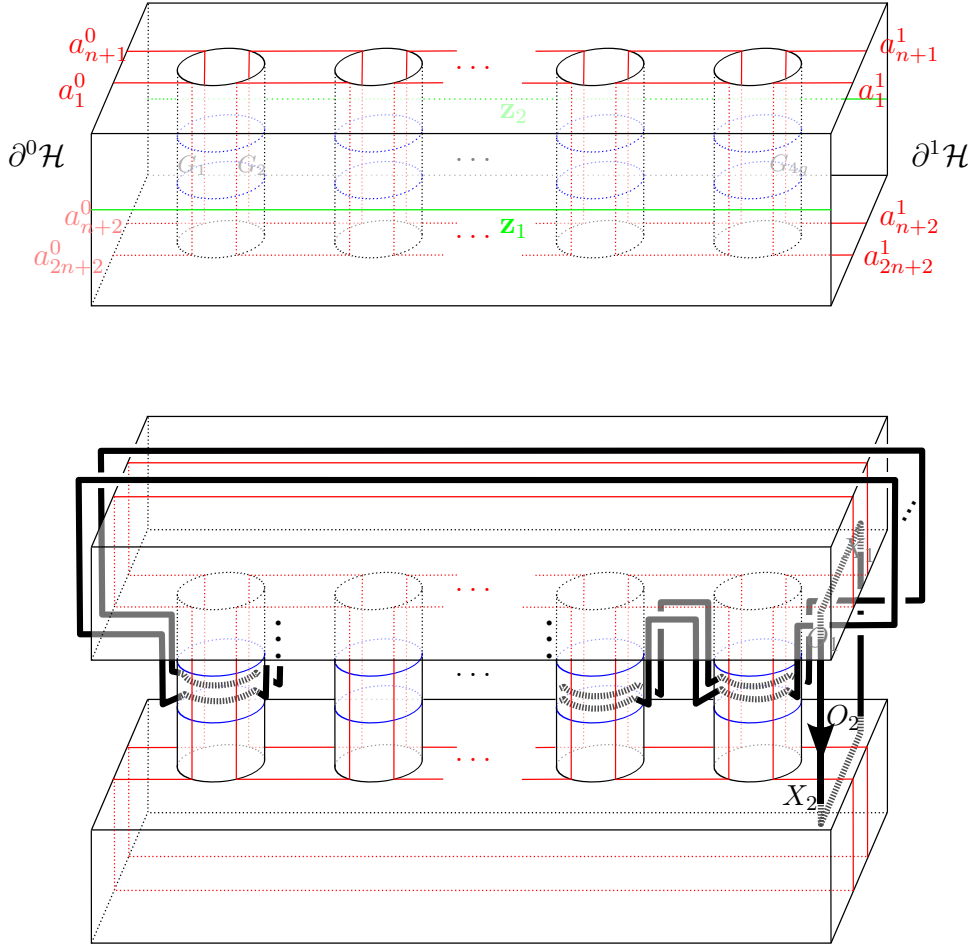} 
       \vskip .1 cm 
       \caption{Top: A diagram for $T$ in $S^2\times I$. There are two grids on each vertical annulus, $X$s and $O$s omitted for simplicity.  Bottom: The corresponding diagram for $T_0$ in $S^3$, along with $T_0$ in bold black.}\label{fig:grid}
\end{figure}

Call the part of $\hc$ away from the four regions resulting from the surgery on $\zzz_i'$ the \emph{nice} part of $\hc$ (this is the part that is the plumbing of grid diagrams).  As in Figure \ref{fig:generic}, we can draw $\hc$ on the plane, as the union of two $2g$-punctured disks  with certain identifications of the boundary, see Figure \ref{fig:hhd}. As seen on Figure \ref{fig:hhd}, we refer to the top/bottom disk as the \emph{top/bottom half} of  $\hc$, respectively. 

Denote the set of generators of $\HH$ by $\mathfrak S(\HH)$, and the subset of generators with $i$ occupied $\alpha$-arcs on the right by $\mathfrak S_i(\HH)$. Denote the subsets of  $\mathfrak S(\HH)$ and  $\mathfrak S_i(\HH)$ that correspond to generators of $\HH^{\circ}$ by $\mathfrak S(\HH)^{\circ}$ and  $\mathfrak S_i(\HH)^{\circ}$, and the corresponding sets of generators of $\HH^{\circ}$ by $\mathfrak S(\HH^{\circ})$ and  $\mathfrak S_i(\HH^{\circ})$, respectively. For a generator $\xx\in \mathfrak S(\HH)$,  denote the corresponding generator in $\mathfrak S(\hc)$ by $\xc$. Define a  \emph{strands} grading on generators by $S(\xc_i)= i$ for $\xc_i\in \mathfrak S_i(\HH^{\circ})$. 

Since $\HH^{\circ}$ is a diagram for $S^3$, any two generators are connected by a domain. Let $\xc, \yc\in \mathfrak S(\HH^{\circ})$, and let $B\in \pi_2(\xc, \yc)$. By adding regions of $\Sigma^{\circ}\setminus \balpha^{\circ}$, we can assume that the domain of $B$ is contained entirely in the top half of $\hc$ and has zero multiplicity in the lowest region (the one containing $O_1$) of the top half of $\hc$.
The oriented boundary of $B$ splits into two pieces,  $\bdy^{\alpha}B\subset \balpha^{\circ}$ and $\bdy^{\beta}B\subset\bbeta^{\circ}$.
The piece $\bdy^{\alpha}B$ is the union of arcs in $\balpha^{\circ}$ such that $\bdy(\bdy^{\alpha}B) = \yc-\xc$. Let $\alpha_1, \ldots, \alpha_{n+1}$ be the $\alpha$-circles in $\HH^{\circ}$ resulting from the gluing of the $\alpha$-arcs in $\HH$, labelled so that $a_i^0\in \alpha_i$, and let $x_i = \xc\cap \alpha_i$, $y_i = \yc\cap \alpha_i$. Below, we turn our attention to the oriented arcs  $c_i:=\bdy^{\alpha}B\cap \alpha_i$, and to each $c_i$ we associate a number $t_i\in \{-1, 0, 1\}$. Since $B$ is contained in the top half of $\hc$, there are three possibilities for each  $c_i$:
\begin{itemize}
\item $c_i$ is contained in the rightmost/leftmost grid if and only if $x_i$ and $y_i$ are (in this case, define   $t_i=0$);
\item $c_i$ covers both the rightmost and the leftmost grid, and is oriented to the right, as seen on Figure \ref{fig:ai}, if and only if $x_i$ is in the rightmost grid and $y_i$ is in the leftmost grid (in this case, define $t_i = 1$);
\item $c_i$ covers both the rightmost and the leftmost grid, and is oriented to the left, as seen on Figure \ref{fig:ai}, if and only if $x_i$ is in the leftmost grid and $y_i$ is in the rightmost grid (in this case, define $t_i = -1$).
\end{itemize}

For $1\leq i\leq n$, let $R_i$ be the region of $\Sigma^{\circ}\setminus(\balpha^{\circ}\cup \bbeta^{\circ})$ containing the image of the interval $(a_i^0, a_{i+1}^0)$ in $\Sigma^{\circ}$, and let $R_{n+1}$ be the topmost region of the top half of $\hc$. See Figure \ref{fig:ai}.
\begin{figure}[h]
 \centering
     \labellist   
        \pinlabel \textcolor{red}{$c_1$} at 170 27
        \pinlabel \textcolor{red}{$c_2$} at 212 37
       \pinlabel \textcolor{red}{$c_n$} at 212 54
       \pinlabel \textcolor{red}{$c_{n+1}$} at 170 68       
        \pinlabel $R_1$ at 194 32 
        \pinlabel $R_n$ at 194 60 
        \pinlabel $R_{n+1}$ at 187 73 
        \pinlabel $X_1$ at 105 75 
        \pinlabel $O_1$ at 105 17 
        \pinlabel $\rotatebox{90}{\textcolor{red}{\ldots}}$ at 193 47
     \endlabellist
       \includegraphics[scale=1.4]{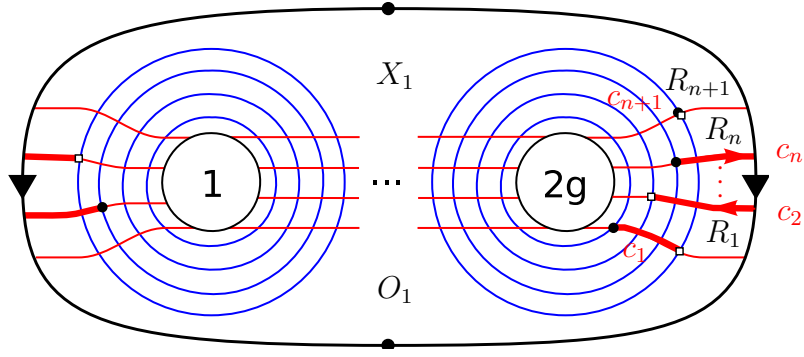} 
       \vskip .1 cm 
       \caption{The top half of $\hc$, along with the arcs $c_i$ in thick red ($c_{n+1}$ is just a point). }\label{fig:ai}
\end{figure}

It is not hard to see that the multiplicity of $B$ at each $R_i$ is $t_1+\cdots+t_i$, and that the multiplicity of $B$ at $R_{n+1}$ is zero if and only if the generators $\xx$ and $\yy$ of $\HH$ corresponding to $\xc$ and $\yc$ occupy the same number of arcs on the right. 

Suppose  $\xc, \yc \in \mathfrak S_i(\HH^{\circ})$. By the above, $\xc$ and $\yc$ are connected by a domain $B$ contained entirely in the nice part of $\hc$, i.e. $n_{X_1}(B) = n_{X_2}(B) = n_{O_1}(B)= n_{O_2}(B) = 0$, so $A_0(\xx)= A_0(\yy)$. Then $B$ is the result of self-gluing a domain $B'$ in $\HH$. Note that the left and right multiplicities of $B'$ match up, i.e. if $p_0\in \bdy^0\HH$ and $p_1\in \bdy^1\HH$ are points that are identified in $\hc$, then $m(\bdy^0B', p_0) = - m(\bdy^1B', p_1)$. Recall the definitions of the sets $S^i_{\OO}, S^i_{\XX}, S^i_{\xx}, S^i_{\yy}$, $\bar{S^i_{\OO}}, \bar{S^i_{\XX}}, \bar{S^i_{\xx}}, \bar{S^i_{\yy}}$,  and of the gradings of domains from \cite[Section 11.2]{pv}. For each subscript $\bullet\in\{\OO, \XX, \xx, \yy \}$, $p_0\in \bar{S^0_{\bullet}}$ if and only if $p_1\in S^1_{\bullet}$. Further, $e(B) = e(B')$, and $n_p(B) = n_p(B')$ for any point $p$. 
Then 
 \begin{align*}
M(B') =& -e(B') - n_{\xx}(B') - n_{\yy}(B') + \frac 1 2 m([\bdy^{\bdy}B'], \bar{S^0_{\xx}} + \bar{S^0_{\yy}} + S^1_{\xx} + S^1_{\yy})  \\
 & - m([\bdy^{\bdy} B'], \bar{\SO^0}+\SO^1) + 2n_{\OO}(B') = -e(B) - n_{\xx}(B) - n_{\yy}(B)  +   2n_{\OO}(B) = M(B)\\ 
A(B') =& \frac 1 2 m([\bdy^{\bdy} B'], \bar{\SX^0} - \bar{\SO^0}+ \SX^1 - \SO^1)+ n_{\OO}(B') - n_{\XX}(B') = n_{\OO}(B) - n_{\XX}(B) = A(B).
\end{align*}
Thus, the relative $(M, A)$ gradings are the same in $\ctt(\HH)$ as in $\cfkhat(\HH^{\circ})$, and the relative $A_0$ grading is zero, i.e. 
 \begin{align*}
M(\xc) - M(\yc) &= M(\xx) - M(\yy) \\
A(\xc) - A(\yc) &= A(\xx) - A(\yy) \\
A_0(\xc) - A_0(\yc) & = 0.
\end{align*}

Next, we compare the gradings of generators with distinct numbers of occupied arcs on the right. The plumbing of $4g$ grid diagrams for $\HH$ corresponds to a sequence of $4g$ shadows $\P_1, \ldots, \P_{4g}$, see \cite[Sections 3 and 4]{pv} and Figure \ref{fig:gens}. Suppose  $\mathfrak S_i(\HH^{\circ})\neq \emptyset$ and  $\mathfrak S_{i+j}(\HH^{\circ})\neq \emptyset$ for some $i,j$, and let $\xx_i\in \mathfrak S_i(\HH)^{\circ}, \xx_{i+j}\in \mathfrak S_{i+j}(\HH)^{\circ}$. The generator $\xx_{i+j}$ has $j$ more strands than $\xx_i$ in each  even-indexed shadow $\P_{2t}$. Choose one strand of $\xx_{i+j}$ in each $\P_{2t}$, and let $p_{2t-1}$ and $p_{2t}$ be the endpoints of this strand in $\bdy^0\P_{2t}$ and $\bdy^1\P_{2t}$, respectively. Replacing these $2g$ strands with the strands from $p_{2t}$ to $p_{2t+1}$ produces a generator $\xx_{i+j-1}\in \mathfrak S_{i+j-1}(\HH)^{\circ}$. Repeating this procedure shows that $\mathfrak S_k(\HH)^{\circ}\neq \emptyset$ for every $i+1\leq k\leq i+j-1$ as well.

Then it suffices to choose two generators  $\xx_i\in \mathfrak S_i(\HH)^{\circ}$ and  $\xx_{i+1}\in\mathfrak S_{i+1}(\HH)^{\circ}$ for each $i$ such that $\mathfrak S_i(\HH)^{\circ}\neq \emptyset$ and $\mathfrak S_{i+1}(\HH)^{\circ}\neq \emptyset$, and understand the relative $(M, A, A_0)$ grading for the corresponding generators $\xc_i, \xc_{i+1}$. Let $\xx_{i+1}\in\mathfrak S_{i+1}(\HH)^{\circ}$ and let $p_0, \ldots, p_{4g}\simeq p_0$  be as above. Modify $\xx_{i+1}$ as follows. 
Let $q_t$ be the topmost point in $\bdy^1\P_t$. The strand at $q_t$ may be contained in $\P_t$ or in $\P_{t+1}$. Starting at $t=1$, and moving up to $t=4g$ (identify $\bdy^1\P_{4g}$ with $\bdy^0\P_1$), do the following exchanges of strands. If $q_t\neq p_t$, take the two distinct strands with ends at $q_t$ and at $p_t$, and exchange their endpoints on $\bdy^0\P_t$. In other words, if one strand connects $p_t$ to another point $p_t'$, and the other strand connected $q_t$ to another point $q_t'$, then replace the two strands with a strand connecting $p_t$ to $q_t'$ and a strand connecting $q_t$ to $p_t'$. In this modified $\xx_{i+1}$, there is a strand connecting $q_{2t-1}$ to $q_{2t}$, for $1\leq t\leq  2g$.
Let $\xx_i$ be the generator obtained from $\xx_{i+1}$ by replacing these strands with strands from $q_{2t}$ to $q_{2t+1}$, as above. Now  $\xx_i\in \mathfrak S_i(\HH)^{\circ}$ and  $\xx_{i+1}\in\mathfrak S_{i+1}(\HH)^{\circ}$ agree almost everywhere, except that $\xx_{i+1}$ contains the strand at the very top of each even-indexed shadow, and $\xx_i$  contains the strand at the very top of each odd-indexed shadow. The Maslov and Alexander gradings on strand generators are defined by counting various intersections of strands, see \cite[Section 3.4]{pv}, and  one sees that $M(\xx_{i+1}) = M(\xx_i)$ and $A(\xx_{i+1}) = A(\xx_i)$.

 Switching back to Heegaard diagrams, $\xx_i$ and $\xx_{i+1}$ differing in the above way is  equivalent to saying that the $4g$-gon $R_{n+1}$ connects $\xc_{i+1}$ to $\xc_i$. Since $e(R_{n+1}) = 1-g$, $n_{\xc_{i+1}}(R_{n+1}) = g/2 = n_{\xc_i}(R_{n+1})$,  $n_{X_1}(R_{n+1}) = 1$ and $n_p(R_{n+1}) = 0$ for any other $p\in \XX\cup \OO$, we see that 
\begin{align*}
M(\xc_{i+1}) - M(\xc_i) &= 1\\
A(\xc_{i+1}) - A(\xc_i) &= 0\\
A_0(\xc_{i+1}) - A_0(\xc_i) &= 1.
\end{align*}

So for $i<j$ and arbitrary $\xc_i\in \mathfrak S_i(\HH^{\circ})$, $\xc_j\in \mathfrak S_j(\HH^{\circ})$, we have
\begin{align}
M(\xc_j) - M(\xc_i) &= M(\xx_j) - M(\xx_i) + j-i\label{eq1}\\
A(\xc_j) - A(\xc_i) &= A(\xx_j) - A(\xx_i)\label{eq2}\\
A_0(\xc_j) - A_0(\xc_i) &= j-i.\label{eq3}
\end{align}

The argument that the isomorphism respects the absolute gradings is analogous to the one from \cite[Section 6]{pv}. With the $4g$ grids arranged as in Figure \ref{fig:gens}, indexed $G_1, \ldots, G_{4g}$ from left to right, let $\xc_{\OO}$ be the generator formed by the bottom-left corner $x_j$ of each $O_j$ in $G_{4i}$ and $G_{4i+1}$, the top-right corner $x_j$ of each $O_j$ in $G_{4i+2}$ and $G_{4i+3}$,  the very top-right corner $x_{4i+1}'$ of each grid $G_{4i+1}$, and the bottom-left corner $x_{4i+3}'$ of each grid $G_{4i+3}$. Define $\xc_{\XX}$ analogously, by replacing $O_j$ with $X_j$  in the above definition. 

\begin{figure}[h]
 \centering
       \includegraphics[scale=1.15]{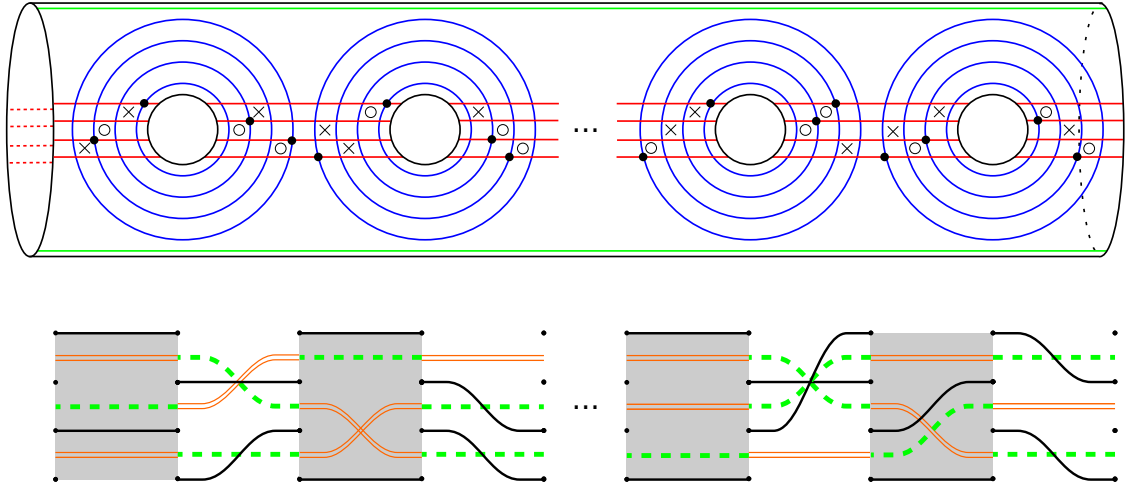} 
       \vskip .1 cm 
       \caption{Top: A Heegaard diagram for $T$ coming from a plumbing of grids and the ``bottom-most" generator $\xx_{\OO}$.
  Bottom: The corresponding sequence of shadows for $T$, and the strands diagram for $\xx_{\OO}$.}\label{fig:gens}
\end{figure}

  Denote the $\beta$ circle containing each $x_i$ or $x_i'$ by $\beta_i$ or $\beta_i'$, respectively. 
Form a set of circles $\bgamma$ by performing handleslides (which are allowed to cross $\XX$ but not $\OO$) of all $\beta_i$ and perturbations of all $\beta_i'$, as in Figure  \ref{fig:hhd}.
We look at the holomorphic triangle map  associated to $(\Sigma, \balpha, \bbeta, \bgamma, \OO)$, see \cite{lmw, osz14, osz5}. Let $k$ be the number of $O$s in $\HH$ (so the number of $O$s in $\hc$ is $k+2$).
Observe that $(\Sigma, \bbeta, \bgamma, \OO)$ is a diagram for $(\#^{k+2}S^1\times S^2)$, and let $\Theta$ be the top-dimensional generator.  Let $\yy$ be the generator of $(\Sigma, \balpha, \bgamma, \OO)$ nearest to $\xc_{\OO}$. There is a holomorphic triangle that maps $\xc_{\OO}\otimes \Theta$ to $\yy$, so $M(\xc_{\OO}) = M(\yy)$.

Observe that $(\Sigma, \balpha, \bgamma, \OO)$
is a diagram for $S^3$ with $k+2$ basepoints, so, as a group graded by the Maslov grading, we have $\hfhat(\Sigma, \balpha, \bgamma, \OO)\cong H_{\ast+k+1}(T^{k+1})$. The diagram has $2^{k+1}$ generators, so they are a basis for the homology.  
Let $\yy'$ be the generator obtained from $\yy$  by replacing the intersection  of $\gamma_{4i+1}'$ and the topmost $\alpha$ of $G_{4i+1}$ with the intersection of $\gamma_{4i+1}'$ and the bottommost $\alpha$ of $G_{4i+2}$, and the intersection  of $\gamma_{4i+3}'$ and the bottommost $\alpha$ of $G_{4i+3}$ with the intersection of $\gamma_{4i+3}'$ and the topmost $\alpha$ of $G_{4i+4}$. There are $k$ disjoint bigons going into $\yy'$, so $M(\yy')\leq -k$.  The shaded $4g$-gon on Figure \ref{fig:hhd} from $\yy'$ to $\yy$ shows that $M(\yy')-M(\yy) = 1$, so $M(\yy)\leq -k-1$. But the $2^{k+1}$ generators are a basis for the homology, so $M(\yy)\in [-k-1, 0]$. Thus, $M(\yy)= -k-1$, so $M(\xc_{\OO})= -k-1$. We can also compute $M(\xx_{\OO})$  using the definition from \cite[Section 3.4]{pv}. The computation is analogous to the one from \cite[Section 6]{pv}, and we see that $M(\xx_{\OO}) = -k$.
Note that $S(\xc_{\OO})=a$, so by Equation \ref{eq1}, for an arbitrary generator $\xc_i\in \mathfrak S_i(\HH^{\circ})$ we have
\[M(\xc_i) - M(\xc_{\OO}) = M(\xx_i) - M(\xx_{\OO}) + i-a\]
so 
\[M(\xc_i) = M(\xx_i) + i-a-1.\]

\begin{figure}[h!]
 \centering
       \labellist   
        \pinlabel {\scriptsize $1$} at 263 183 
        \pinlabel {\scriptsize  $1$} at 263 119 
        \pinlabel {\scriptsize  $2$} at 270 69 
        \pinlabel {\scriptsize  $2$} at 270 18
     \endlabellist
       \includegraphics[scale=1.36]{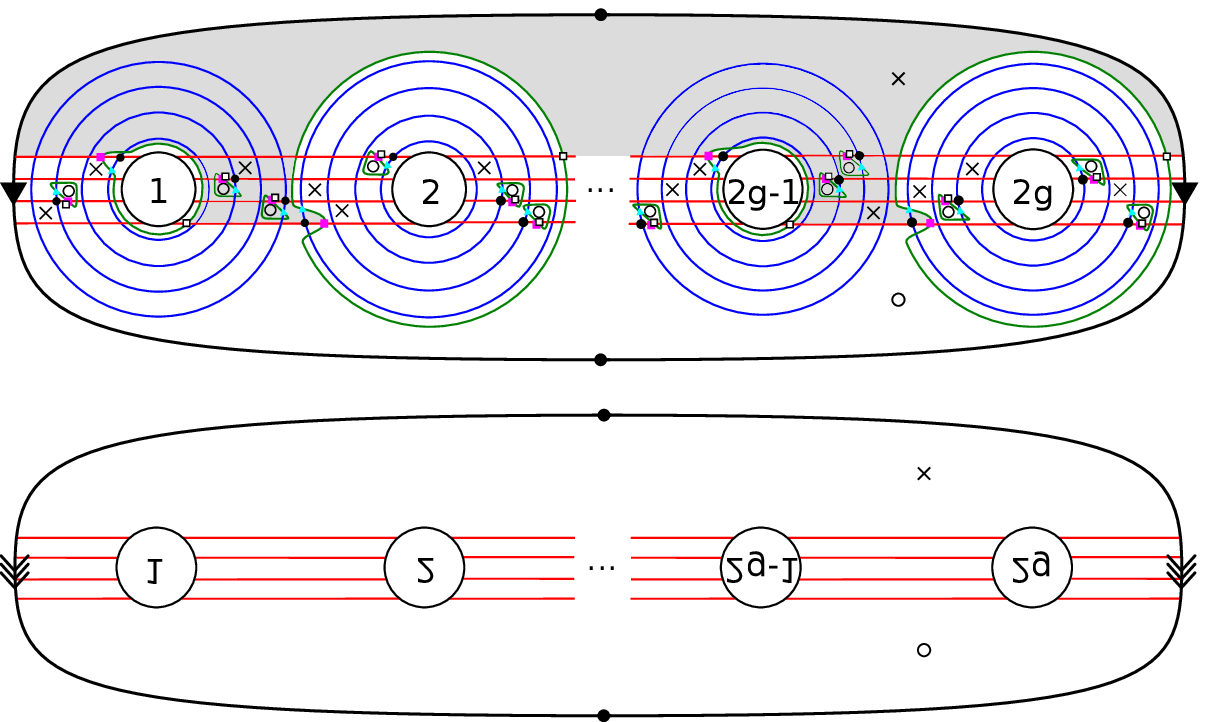} 
       \vskip .1 cm 
       \caption{The Heegaard triple $(\Sigma, \balpha, \bbeta, \bgamma, \OO)$, with the original $\XX$ markings left in. The black dots form the generator $\xx_{\OO}$, the purple squares form $\yy$, the white squares from $\yy'$, and the cyan triangles form $\Theta$.}\label{fig:hhd}
\end{figure}

Similarly, for the $z$-normalized (or, $\XX$-normalized, to match the notation in this paper) grading $N$, we have $N(\xc_{\XX}) = -k-1$. Since $N = M - 2A - (k+2 - l)$, we get 
\[A(\xc_{\XX}) = \frac 1 2 (M(\xc_{\XX}) - N(\xc_{\XX}) - (k+2 - l)) = \frac 1 2 (M(\xc_{\XX})+ l-1).\]

Again using the definition from \cite[Section 3.4]{pv} as we do in \cite[Section 6]{pv}, we see that $M(\xx_{\OO}) = -k$, $N(\xx_{\XX}) = -k$, $A(\xx_{\XX}) = \frac 1 2 M(\xx_{\XX})$. Since $S(\xx_{\XX})=n-a$, we have $M(\xc_{\XX}) = M(\xx_{\XX}) + n-2a-1$, so we get
\[A(\xc_{\XX}) = A(\xx_{\XX}) +\frac 1 2( n-2a+ l-2).\]

The Alexander multigrading on a generator $\xc$ can be described by  the relative  $\textrm{Spin}^c$ structure $\mathfrak s(\xc)\in \underline{\textrm{Spin}}^c(S^3, L)$, see \cite{oszlink}. In the case when the link is in $S^3$, one can think of  $A_i$ by looking at the projection of a Seifert surface for $L_i$ onto $\hc$. Specifically, for a generator $\xc_i \in\mathfrak S_i(\hc)$, we can compute its $A_0$ grading in the following way. Connect $X_1$ to $O_1$ and $X_2$ to $O_2$ away from $\bbeta^{\circ}$, and $O_1$ to $X_2$ and $O_2$ to $X_1$ away from $\balpha^{\circ}$ to obtain a  curve $C$ on $\hc$ representing $L_0$, so that  $C$ is negative the boundary of a disk $D$ that is a neighborhood of the rightmost grid (in general $C$ may be immersed but not necessarily embedded). Then 
\[A_0(\xc_i) = \frac 1 2(e(D)+ 2n_{\xc_i}(D) - n_{\XX}(D) - n_{\OO}(D)) =  i -\frac{n+1}{2}.\]
It follows that 
\[A'(\xc_i)= A(\xx_i) -S(\xx_i)+ \frac l 2  +n-a-1\]

This completes the identification of gradings.
\end{proof}

\remark The authors are in the process of  upgrading the invariants in \cite{pv} to have Alexander multigradings, corresponding to different components of the tangle.  The arguments in this paper automatically imply that the isomorphism from Theorem \ref{thm:graded} respects the multigrading, with appropriate additive constants.  

 Last, we prove Corollary \ref{cor:braid}.
 
 \begin{proof}[Proof of Corollary \ref{cor:braid}]  
Fix $n$, and let $\mathbbm 1_n$ denote the trivial braid on $n$ strands (oriented from top to bottom). Let $B$ be an $n$-braid. 
Let $\Aa=\Aa(\partial^1\mathbbm 1_n)=\Aa(\partial^1B)$, and let $\textrm{H}(\textrm{Mod}_{\Aa})$ be the homotopy category of  right type $A$ modules over $\Aa$. Suppose that $-\boxtimes \ctt(B)$ acts as the identity on  $\textrm{H}(\textrm{Mod}_{\Aa})$. Recall that one can recover the homotopy type of any $\AA$ bimodule, i.e. $A_{\infty}$ bimodule,  ${}_{\Aa}N_{\Aa}$ from the functor $-\widetilde{\otimes} N$ on $\textrm{H}(\textrm{Mod}_{\Aa})$, for example as  $Q_{-\widetilde{\otimes} N}(\Aa_\Aa,\Aa_\Aa)$ 
in \cite[(7.23) and (7.24)]{seidel}. Thus,  $-\boxtimes \ctt(B) = -\widetilde{\otimes} (\Aa\boxtimes \ctt(B))$ being the identity implies $\Aa\simeq \Aa\boxtimes \ctt(B)$, so $ \ctt(\mathbbm 1_n)\simeq\ctt(B)$.
We show that the latter implies $B=\mathbbm 1_n$. Let $k\geq 1$ be an integer such that $B^k$ is pure. Then $\ctt(B^k)\simeq \ctt(B)\boxtimes\cdots\boxtimes \ctt(B)\simeq \ctt(\mathbbm 1_n) \boxtimes\cdots\boxtimes \ctt(\mathbbm 1_n)\simeq\ctt(\mathbbm 1_n^k)\simeq \ctt(\mathbbm 1_n)$. Taking Hochschild homology, it follows that $\widehat{\mathit{HFL}}((B^k)_0,0)\cong \widehat{\mathit{HFL}}((\mathbbm 1_n)_0,0)$. 
Recall that, in our notation, \cite[Theorem 1(b)]{bg} says that if there is a triply graded isomorphism $\widehat{\mathit{HFL}}(T_0,0)\cong \widehat{\mathit{HFL}}((\mathbbm 1_n)_0,0)$ and $T$ is a pure braid, then $T=\mathbbm 1_n$. Thus, $B^k = \mathbbm 1_n$. Since the braid group is torsion-free, it follows that $B=\mathbbm 1_n$.
\end{proof}

\bibliographystyle{hamsplain2}
\bibliography{master}

\end{document}